\documentclass[11pt]{amsart}
\usepackage{graphicx}
\usepackage{stmaryrd}
\usepackage{amssymb}
\usepackage{ytableau}
\usepackage{mathtools,color}
\usepackage[alphabetic]{amsrefs}
\usepackage{tikz}
\usepackage{tikz-3dplot}
\usepackage{eucal}
\usepackage{float}
\usepackage{fullpage}
\usepackage{url}
\usepackage[colorlinks = true,
linkcolor = blue,
urlcolor  = blue,
citecolor = black,
anchorcolor = black]{hyperref}
\usepackage[T1]{fontenc}

\title{The structure and normalized volume of Monge polytopes}

\author{William Q. Erickson}
\address{
William Q.~Erickson\\
Department of Mathematics\\
Baylor University \\ 
One Bear Place \#97328\\
Waco, TX 76798} 
\email{Will\_Erickson@baylor.edu}

\author{Jan Kretschmann}
\address{
Jan Kretschmann\\
Department of Mathematical Sciences\\
University of Wisconsin--Milwaukee \\ 
3200 N.~Cramer St.\\
Milwaukee, WI 53211} 
\email{kretsc23@uwm.edu}

\theoremstyle{plain}
\newtheorem{theorem}{Theorem}[section]
\newtheorem{prop}[theorem]{Proposition}

\newtheorem{lemma}[theorem]{Lemma}
\newtheorem{corollary}[theorem]{Corollary}

\theoremstyle{definition}
\newtheorem{definition}[theorem]{Definition}
\newtheorem{rem}[theorem]{Remark}
\newtheorem{ex}[theorem]{Example}

\newcommand{\M}{\mathcal{M}}
\newcommand{\SM}{\mathcal{SM}}
\newcommand{\AM}{\mathcal{HM}}
\renewcommand{\P}{\mathcal{P}}

\newcommand{\NE}[1]{\lceil #1 \rceil}
\newcommand{\SW}[1]{\lfloor #1 \rfloor}
\newcommand{\HOR}[1]{[ #1 \bullet ]}
\newcommand{\VER}[1]{[ \bullet #1 ]}
\newcommand{\HV}[1]{\llbracket #1 \rrbracket}
\newcommand{\NESW}[1]{[ #1 ]}



\begin{document}

\begin{abstract}
    A matrix $C$ has the \emph{Monge property} if $c_{ij} + c_{IJ} \leq c_{Ij} + c_{iJ}$ for all $i < I$ and $j < J$.
    Monge matrices play an important role in combinatorial optimization; for example, when the transportation problem (resp., the traveling salesman problem) has a cost matrix which is Monge, then the problem can be solved in linear (resp., quadratic) time.
    For given matrix dimensions, we define the \emph{Monge polytope} to be the set of nonnegative Monge matrices normalized with respect to the sum of the entries.
    In this paper, we give an explicit description and enumeration of the vertices, edges, and facets of the Monge polytope; these results are sufficient to construct the face lattice.
    In the special case of two-row Monge matrices, we also prove a polytope volume formula.
    For \emph{symmetric} Monge matrices, we show that the Monge polytope is a simplex and we prove a general formula for its volume.
\end{abstract}

\maketitle

\section{Introduction}

\subsection{Monge matrices in combinatorial optimization}

Two of the best-known combinatorial optimization problems are the \emph{transportation problem} (TP) and the \emph{traveling salesman problem} (TSP).
The former is solvable in polynomial time while the latter is famously NP-hard.
But in the special case where the cost matrix $C$ has the \emph{Monge property}, i.e.,
\[
c_{ij} + c_{IJ} \leq c_{Ij} + c_{iJ} \quad \text{for all $i<I$ and $j < J$,}
\]
the TP can be solved by a greedy algorithm (the ``northwest corner rule'') in linear time, while the TSP can be solved by finding a pyramidal tour in $O(n^2)$ time.
These results are due to Hoffman~\cite{hoffman} and Gilmore--Lawler--Shmoys~\cite{Gilmore}, respectively.
(Hoffman named the property after the 18th-century geometer Gaspard Monge, who, apart from running a rather famous pyramidal tour of his own,\footnote{Several sources relate that on a hot summer day during his campaign in Egypt, Napoleon Bonaparte proposed that his companions compete in a race to ascend a pyramid.
The winner, evidently, was the 53-year old Monge.} first explored the property in his treatise~\cite{Monge} launching the field of optimal transport theory.)
Moreover, when the TSP has a \emph{symmetric} Monge matrix, as is natural in many applications,  Supnick~\cite{Supnick} showed that an optimal solution is a fixed tour which does not even depend on the actual values of the matrix.

The TP and TSP are just two of the many optimization problems whose solution time lessens drastically (or even becomes trivial) when the underlying structure is given by a Monge matrix.
We refer the reader to the excellent comprehensive survey~\cite{Burkard} devoted to the multitude of applications of Monge matrices, a list which has only grown since its publication; see also Villani's monumental reference~\cite{Villani}.

\subsection{Motivation}

It is not hard to see that the set of $p \times q$ Monge matrices (or symmetric $n \times n$ Monge matrices) with nonnegative entries forms a pointed cone, whose extremal rays were described by Rudolf and Woeginger~\cite{Rudolf}.
With this as our starting point, our present paper is founded on the following simple observation: the optimal solution to the TP or TSP is invariant under multiplying the cost matrix $C$ by scalars $\lambda \in \mathbb{R}_{>0}$.
This is because in either problem, the solution is simply a matrix $T$ (called a \emph{transportation plan} or \emph{flow matrix} in the literature) which minimizes the Hadamard product $\sum_{i,j} c_{ij} t_{ij}$.
In other words, in the cone of Monge matrices, one can view every ray through $0$ as an equivalence class, since all the points (i.e., cost matrices) on that ray yield the same solution to the relevant optimization problem.
It seems natural to eliminate this redundancy by normalizing: hence in this paper, we study the intersection of the Monge cone with the probability simplex (i.e., where all matrix entries are nonnegative and sum to $1$).
The result is a bounded polytope, which we call the \emph{Monge polytope}.
From this perspective, the points in the Monge polytope are precisely the \textit{distinct} cost matrices with respect to the TP, the TSP, and other optimization problems.
We define analogously the \emph{symmetric Monge polytope} and the \emph{hollow symmetric Monge polytope} (where a \emph{hollow} matrix has all zeros on its diagonal).
All three of these polytopes correspond to cost matrices which arise naturally in different variants of the TP and TSP and many other problems.

We emphasize that the polytopes in this paper are quite different from those studied in recent work applying convex geometry to optimal transport theory; see~\cite{Friesecke} and~\cite{Vogler} and the references therein.
The authors cited above studied \emph{Monge} and \emph{Kantorovich polytopes} in which the points are \emph{solutions}, rather than cost functions, for transport problems.

\subsection{Main results}

The main result of this paper is a description of the structure of the various Monge polytopes, along with their \emph{normalized volume}, i.e., the proportion of the probability simplex which they occupy.
We devote Section~\ref{sec:symmetric results} to the symmetric case.
It turns out (Theorem~\ref{thm:prob HM}) that the hollow symmetric (resp., the symmetric) Monge polytopes are simplices, with (respective) normalized volumes
\[
\frac{1}{{\rm sf}(n-1)^2} \quad \text{and} \quad \frac{1}{{\rm sf}(n-1)^2 \cdot n^n},
\]
where ${\rm sf}(n) \coloneqq 1!2!3!\cdots n!$ is the \emph{superfactorial} as defined by Sloane and Plouffe~\cite{Sloane}.
We emphasize (Section~\ref{sub:metric face}) that the hollow symmetric Monge polytope has a distinguished face consisting precisely of those cost matrices which induce a metric space structure in the context of the TP or TSP. 

In contrast with the symmetric Monge polytopes, the polytope of generic $p \times q$ Monge matrices has a much more complicated structure, which we describe in Section~\ref{sec:generic results}.
In particular, we give an explicit description and enumeration of its vertices and facets (Proposition~\ref{prop:Mpq vertices} and Theorem~\ref{thm:facets generic}).
Using these results, one can program the vertex--facet incidence matrix, from which the face lattice can be constructed.
We also give an explicit description and enumeration of the edges (Theorem~\ref{thm:edges generic}).
Finally, we take a first step toward a general formula for the normalized volume, by showing (Theorem~\ref{thm:volume 2-row}) that for $2 \times p$ matrices this volume is $1/p!$.

\subsection{Future research and open problems}

The results and methods in this paper suggest some further problems of a combinatorial nature.
We are interested, for instance, in finding a bijective proof of the fact mentioned below in Remark~\ref{rem:OEIS vertices square}: for $p \times p$ Monge matrices, the number of polytope vertices equals the number of regions obtained by intersecting $p$ ellipses in the plane.
The most obvious open problem is to find a general volume formula for the generic Monge polytope, beyond the two-row case (see the closing paragraph of the paper).
Our proof of the volume formula in the two-row case involves our writing down a Stanley decomposition for the monoid of Monge matrices with nonnegative integer entries; it seems that such a decomposition (for arbitrary matrix dimensions) might be of intrinsic interest in applications, since it leads to a canonical way of writing Monge matrices as linear combinations of vertices.
Finally, since a cost matrix with the Monge property improves computation times so dramatically, it seems worth studying whether the characterization of the Monge polytope in this paper may be useful in approximating a non-Monge matrix by the ``nearest'' Monge matrix, in the appropriate sense.

\section{Preliminaries}

\subsection{Convex polytopes}

The following terminology and results can be found in any standard reference on polytopes; see, for example,~\cite{Ziegler}*{Ch.~1--2}.

Consider a finite subset $W\subseteq \mathbb{R}^d$. 
The \emph{convex hull} of $W$ defines a polytope in $\mathbb{R}^d$:
\[
\mathcal{P} = \mathrm{conv}(W) \coloneqq \left\{\sum_{w \in W} \lambda_w w : \lambda_w \geq 0, \: \sum_{w \in W} \lambda_w = 1\right\}.
\]
The \emph{dimension} of $\mathcal{P}$ is the dimension of its affine span. 
If, in addition, each $w \in W$ cannot be expressed as a convex combination of any other points in $\mathcal{P}$, then we call the elements of $W$ the \emph{vertices} of $\mathcal{P}$, and we call $W$ the \emph{vertex set} of $\mathcal{P}$, denoted by $V(\mathcal{P}) = W$.
This description of a polytope $\mathcal{P}$ in terms of its vertices is called the \emph{$V$-representation} of $\mathcal{P}$.
A \emph{face} is a subset $F\subseteq \mathcal{P}$ with the following property: if $x,y\in \mathcal{P}$, then $\lambda x+(1-\lambda)y \in F$ implies $x,y \in F$ for $0 \leq \lambda \leq 1$.
A \emph{vertex} is thus a $0$-dimensional face.
An \emph{edge} is a $1$-dimensional face, while a \emph{facet} is a face whose dimension is one less than $\dim \mathcal{P}$.
The \emph{$f$-vector} of $\P$ is the vector $(f_0, f_1, \ldots, f_d)$, where $f_i$ denotes the number of $i$-dimensional faces of $\P$.

Alternatively, consider a \emph{half-space}, i.e., the set of points lying weakly to one side of an affine hyperplane in $\mathbb{R}^d$.
The bounded intersection of finitely many half-spaces is a convex polytope.
This description of a polytope $\mathcal{P}$ in terms of a finite system of linear inequalities (corresponding to the system of half spaces whose intersection defines $\mathcal{P}$) is called the \emph{$H$-representation} of $\mathcal{P}$.
From this viewpoint, a \emph{face} of $\mathcal{P}$ is defined as a nonempty set of points at which some subsystem of the defining inequalities is \emph{tight}, i.e., the inequalities are actually equalities.
(This can be restated as saying that a face is the intersection of $\P$ with certain of its supporting hyperplanes.)
In particular, supposing that the defining system of inequalities is minimal (i.e., contains no redundant inequalities), a \emph{facet} consists of those points at which exactly one of the defining inequalities is tight.
It is a standard fact that the two definitions of a face (one for the $V$-representation, the other for the $H$-representation) are in fact equivalent.

Recall that a \emph{simplex} is a polytope $\mathcal{P}$ such that $\#V(\mathcal{P}) = \dim \mathcal{P} + 1$.
In this case, every nonempty subset of $V(\mathcal{P})$ defines a face of $\mathcal{P}$ via its convex hull.

\subsection{Three types of Monge polytopes}

A real matrix $C$ is said to have the \emph{Monge property} if it satisfies the inequalities
\begin{equation}
    \label{Monge defining inequality}
    c_{ij} + c_{IJ} \leq c_{iJ} + c_{Ij} \quad \text{for all $i < I$ and $j < J$}.
\end{equation}
Throughout the paper, we will use the letter $C$ when referring to a Monge matrix (this convention has arisen from the \emph{cost} matrix in an optimization problem).
As mentioned in the introduction, we wish to study the Monge matrices modulo multiplication by scalars.
Therefore we will restrict our attention to the probability simplex 
\begin{equation}
\label{Ppq}
    \P_{p \times q} \coloneqq \Big\{ A \in \mathbb{R}_{\geq 0}^{p \times q} : \sum_{i,j} a_{ij} = 1 \Big\}
\end{equation}
inside the space of matrices.
Note that $\P_{p \times q}$ is actually the polytope defined by the inequalities
\begin{equation}
    \label{nonnegative}
    c_{ij} \geq 0 \quad \text{for all $i,j$},
\end{equation}
intersected with the affine hyperplane defined by the sum-one condition.
We define the \emph{Monge polytope} to be
\[
\M_{p \times q} \coloneqq \P_{p \times q} \cap \{\text{Monge matrices} \},
\]
which is a polytope of dimension $pq-1$, defined by the inequalities~\eqref{Monge defining inequality} and contained in the affine hyperplane containing $\P_{p \times q}$.

It is natural in many applications to restrict our attention to $n \times n$ symmetric Monge matrices (also called \emph{Supnick matrices}), by imposing the conditions
\begin{equation*}
    c_{ij} = c_{ji} \quad \text{for all $1 \leq i \leq j \leq n$.}
\end{equation*}
In certain applications it is necessary to restrict one's attention even further to \emph{hollow} symmetric matrices, i.e., matrices with all zeros on the diagonal:
\begin{equation*}
    c_{ii} = 0 \quad \text{for all $1 \leq i \leq n$}.
\end{equation*}
We will therefore also consider the following subsimplices of the probability simplex:
\begin{align}
\label{HP and SP}
\begin{split}
    \mathcal{HP}_{n} &\coloneqq \P_{n \times n} \cap \{ \text{hollow symmetric matrices} \},\\
    \mathcal{SP}_{n} &\coloneqq \P_{n \times n} \cap \{ \text{symmetric matrices} \}.
    \end{split}
\end{align}
Then we define the \emph{hollow symmetric} (resp., the \emph{symmetric}) \emph{Monge polytope} just as before:
\begin{align*}
    \AM_n &\coloneqq \mathcal{HP}_n \cap \{ \text{Monge matrices}\},\\
    \SM_n &\coloneqq \mathcal{SP}_n \cap \{ \text{Monge matrices}\}.
\end{align*}
It is clear that $\AM_{n}$ is a polytope of dimension $\frac{n(n-1)}{2}-1$, while $\SM_n$ is a polytope of dimension $\frac{n(n+1)}{2} - 1$.
This follows from taking the number of degrees of freedom in the $n \times n$ matrix, and subtracting one to account for the intersection with the affine hyperplane containing $\P_{n \times n}$.

Note that the Monge property~\eqref{Monge defining inequality} gives the H-representation of the Monge polytopes.
The first goal in this paper will be to obtain the V-representation (see Propositions~\ref{prop:SM vertices} and~\ref{prop:Mpq vertices}).

\subsection{Asymptotics of rational generating functions}

Given a rational function $f(t)$, its Taylor expansion yields a formal power series $f(t) = \sum_{k = 0}^\infty f_k t^k$, which is called the \emph{generating function} of the sequence $\{f_k\}_{k=0}^\infty$ of coefficients.  
By viewing $f(t)$ as a meromorphic function, one can obtain asymptotic information about the coefficients $f_k$.
We record the following standard result in analytic combinatorics, which follows from~\cite{FS}*{Thm.~IV.9}, with proof given in~\cite{Melczer}*{Ch.~12, corollary to Lemma~3}:

\begin{lemma}
    \label{lemma:asymptotic}
    Let $N(t)$ and $D(t)$ be polynomials, with $\deg N < \deg D$ and with $D(0)=1$.
    Let $f(t) \coloneqq \frac{N(t)}{D(t)} = \sum_{k=0}^\infty f_k t^k$.
    Let $\alpha \in \mathbb{C}$ be a root of $D(t)$ but not of $N(t)$, such that $|\alpha| \leq |\beta|$ for all roots $\beta$ of $D(t)$ over the complex numbers.
    Let $m$ denote the multiplicity of $\alpha$ as a root of $D(t)$.
    Suppose that $m$ is greater than the multiplicity of $\beta$ for all roots $\beta$ such that $|\alpha| = |\beta|$.
    Then we have
    \[
    f_k = (-1)^{m} \cdot \frac{m N(\alpha)}{\alpha^{m+k} D^{(m)}(\alpha)} k^{m-1} + O(k^{m-2}).
    \]
\end{lemma}

In this paper, we will encounter only the special case where $D(t) = \prod_{i=1}^m (1-t^{p_i})$ for positive integers $p_i$.
Then we have $\alpha = 1$ with multiplicity $m$, which is maximal among the multiplicities of all other roots of $D(t)$, which all have modulus 1.
After applying the Leibniz rule $m$ times to $D(t)$, the only terms that do not vanish at $t=1$ are the $m!$ terms of the form $\prod_{i=1}^m \frac{d}{dt}(1-t^{p_i}) = (-1)^m \prod_{i=1}^m p_i$.
 Upon simplifying, we obtain the following specialization of Lemma~\ref{lemma:asymptotic}: 

\begin{corollary}
\label{cor:asymptotic special}
    Assume the hypotheses of Lemma~\ref{lemma:asymptotic}, in the special case $D(t) = \prod_{i=1}^m (1-t^{p_i})$ for positive integers $p_i$.
    Then we have
    \[
    f_k = \frac{N(1)}{(m-1)! \prod_{i=1}^m p_i} k^{m-1} + O(k^{m-2}).
    \]
\end{corollary}

Another useful fact will be the asymptotic behavior of the binomial coefficients
\begin{equation}
\label{composition asymptotic}
\binom{k+\ell}{\ell} \coloneqq \frac{(k+\ell)!}{k!\ell!} = \frac{(k+\ell) (k + \ell -1) \cdots (k+1)}{\ell!} = \frac{k^\ell}{\ell!} + O(k^{\ell-1}). 
\end{equation}

\section{Notation for special matrices}

It turns out that the vertices of the Monge polytopes are matrices with very particular block forms.
Since these vertices play a central role in expressing our results, we will denote them using symbols that resemble the appearance of the matrices.

Let $\mathbf{1}_{a \times b}$ be the $a \times b$ matrix whose entries are all 1's.
Given $p>a$ and $q>b$, we define the following $p \times q$ matrices, written below as $2 \times 2$ block matrices:

\[
    \NE{a \times b} \coloneqq \begin{bmatrix}
                0 & \mathbf{1}_{a \times b} 
                \\ 0 & 0
            \end{bmatrix}, \qquad
            \SW{a \times b} \coloneqq \begin{bmatrix}
                0 & 0
                \\ \mathbf{1}_{a \times b} & 0
            \end{bmatrix}.
        \]
        Next, let $\mathbf{e}_i\in \mathbb{R}^p$ be the column vector with 1 in the $i$th coordinate and 0's elsewhere.
        Let $\mathbf{e}_j^T \in \mathbb{R}^q$ be the row vector with 1 in the $j$th coordinate and 0's elsewhere.
        Then we define the following $p \times q$ matrices:
        \[
        \HOR{i} \coloneqq \begin{bmatrix}
            \mathbf{e}_i \cdots \mathbf{e}_i
        \end{bmatrix}, \qquad 
        \VER{j} \coloneqq \begin{bmatrix}
            \mathbf{e}^T_j \\
            \vdots \\
            \mathbf{e}^T_j
        \end{bmatrix}.
\]
For example, if $p=3$ and $q=4$, then we have
\begin{align*}
\NE{2 \times 1} &= \begin{bmatrix}
    0 & 0 & 0 & 1\\
    0 & 0 & 0 & 1\\
    0 & 0 & 0 & 0
\end{bmatrix}, &
\qquad 
\SW{2 \times 3} &= \begin{bmatrix}
    0 & 0 & 0 & 0\\
    1 & 1 & 1 & 0\\
    1 & 1 & 1 & 0
\end{bmatrix},\\[2ex]
\HOR{3} &= \begin{bmatrix}
    0 & 0 & 0 & 0\\
    0 & 0 & 0 & 0\\
    1 & 1 & 1 & 1
\end{bmatrix}, &
\qquad
\VER{3} &= \begin{bmatrix}
    0 & 0 & 1 & 0\\
    0 & 0 & 1 & 0\\
    0 & 0 & 1 & 0
\end{bmatrix}.
\end{align*}
When we treat symmetric matrices, we will write $n$ for the common value $p = q$.
We use the following shorthand:
\[
\NESW{a \times b} \coloneqq \NE{a \times b} \; + \, \SW{b \times a}, \qquad 
\HV{i} \coloneqq \HOR{i} \; + \VER{i}.
\]
For example, if $n=6$, then we have
\[
\NESW{3 \times 2} = \begin{bmatrix}
    0 & 0 & 0 & 0 & 1 & 1\\
    0 & 0 & 0 & 0 & 1 & 1\\
    0 & 0 & 0 & 0 & 1 & 1\\
    0 & 0 & 0 & 0 & 0 & 0\\
    1 & 1 & 1 & 0 & 0 & 0\\
    1 & 1 & 1 & 0 & 0 & 0
\end{bmatrix}
 \qquad \text{and} \qquad
 \HV{5} = \begin{bmatrix}
    0 & 0 & 0 & 0 & 1 & 0\\
    0 & 0 & 0 & 0 & 1 & 0\\
    0 & 0 & 0 & 0 & 1 & 0\\
    0 & 0 & 0 & 0 & 1 & 0\\
    1 & 1 & 1 & 1 & 2 & 1\\
    0 & 0 & 0 & 0 & 1 & 0
\end{bmatrix}.
\]
We will write a hat symbol to denote projection onto the probability simplex $\P_{p \times q}$ or $\P_{n \times n}$, as defined in~\eqref{Ppq}:
\begin{align*}
    \widehat{\NE{a \times b}} &\coloneqq \frac{1}{ab}\NE{a \times b}, &
    \widehat{\SW{a \times b}} &\coloneqq \frac{1}{ab}\SW{a \times b},\\[2ex]
    \widehat{\HOR{i}} &\coloneqq \frac{1}{q}\HOR{i}, &
    \widehat{\VER{j}} &\coloneqq \frac{1}{p}\VER{j},\\[2ex]
    \widehat{\NESW{a \times b}} &\coloneqq \frac{1}{2ab}\NESW{a \times b}, &
    \widehat{\HV{i}} &\coloneqq \frac{1}{2n}\HV{i}.
\end{align*}

\begin{rem}
    We have the following dictionary between our notation above and that of Rudolf and Woeginger~\cite{Rudolf}*{\S2}:
    \begin{align*}
        \NE{a \times b} & =  R^{(a, q+b-1)}, &
        \SW{a \times b} & = L^{(p-a+1, b)},\\
        \HOR{i} & = H^{(i)}, &
        \VER{j} & = V^{(j)},\\
        \NESW{a \times b} & = T^{(n-a+1, b)}, &
        \HV{i} & = S^{(i)}.
    \end{align*}
\end{rem}

\section{Main results: symmetric Monge polytopes}
\label{sec:symmetric results}

\subsection{Structure of the polytopes}

We begin by writing down the vertices of the hollow symmetric and symmetric Monge polytopes:

\begin{prop}
\label{prop:SM vertices}
    We have
    \begin{align*}
    V(\AM_n) &= \left\{ \widehat{\NESW{a \times b}} : a+b \leq n \right\}, \\[2ex]
    V(\SM_n) &= V(\AM_n) \cup \left\{ \widehat{\HV{i}} : 1 \leq i \leq n \right\}.
    \end{align*}
    Therefore $\#V(\AM_n) = \frac{n(n-1)}{2}$, and $\#V(\SM_n) = \frac{n(n+1)}{2}$.
\end{prop}

\begin{proof}
    It is clear that the matrices in the (claimed) vertex sets all belong to $\AM_n$ (resp., $\SM_n$).
    Therefore any convex combination of these matrices also lies in $\AM_n$ (resp., $\SM_n$).
    We will show that these matrices are in fact the vertices of the respective polytope, by proving something stronger: namely, each element of the polytope can be written as a \textit{unique} convex combination of these matrices.
    
    Let $\{\lambda_{ab}\}_{a+b \leq n}$ be a set of nonnegative real numbers summing to 1, and put
    \begin{equation}
    \label{C sum lambdas}
        C = \sum_{a+b \leq n} \lambda_{ab} \, \widehat{\NESW{a \times b}} \in \AM_n.
    \end{equation}
    Then we have
    \[
    c_{i,n-j+1} = c_{n-j+1,i} = \sum_{\substack{a \geq i,\\b \geq j\phantom{,}}} \frac{\lambda_{ab}}{2ab}.
    \]
    This suggests the following (unique) way to recover the coefficients $\lambda_{ab}$ given any $C \in \AM_n$.
    For each pair $a + b = n$, the coefficient $\lambda_{ab}$ is just $2ab$ times the superdiagonal entry $c_{a,n-b+1}$.
    Then one obtains a new symmetric matrix
    \[
        C' = C - \sum_{a+b=n}\lambda_{ab} \, \widehat{\NESW{a \times b}}
    \]
    with zeros on its superdiagonal.
    For each $a+b=n-1$, one sees that $\lambda_{ab}$ is $2ab$ times the entry $c'_{a,n-b+1}$.
    Repeating this process gives a unique way to obtain all of the coefficients $\lambda_{ab}$.

    For $\SM_n$, the same argument holds if we begin with the diagonal elements of $C \in \SM_n$.
    In particular, include the additional coefficients $\lambda_i$ for each $1 \leq i \leq n$, corresponding to the matrices $\widehat{\HV{i}}$.
    Each $\lambda_i$ is recovered by multiplying the entry $c_{ii}$ by $2n$.
    Subtracting from $C$ the matrices $\lambda_i \, \widehat{\HV{i}}$ yields a hollow symmetric matrix, from which the coefficients $\lambda_{ab}$ can be recovered as described in the $\AM_n$ case above.
\end{proof}

\begin{prop}
    Both $\AM_n$ and $\SM_n$ are simplices, of dimension $\frac{n(n-1)}{2}-1$ and $\frac{n(n+1)}{2}-1$, respectively.
\end{prop}

\begin{proof}
    In both cases, it is immediate from Proposition~\ref{prop:SM vertices} that the number of vertices is one more than the dimension of the polytope.
    Recall that the dimension was obtained by taking the number of degrees of freedom in a symmetric (or hollow symmetric) matrix, and subtracting one because the polytope also satisfies the defining equality of the affine hyperplane containing $\mathcal{P}_{n \times n}$.
\end{proof}

\subsection{Normalized volume of $\AM_n$ and $\SM_n$}

We now turn to the normalized volume of the symmetric Monge polytopes.
Recall from~\eqref{HP and SP} the subsimplices $\mathcal{HP}_n$ and $\mathcal{SP}_n$, consisting of (hollow) symmetric $n \times n$ matrices with nonnegative entries summing to 1.
Let ${\rm Vol}(\mathcal{P})$ denote the $d$-dimensional volume of a $d$-dimensional polytope $\mathcal{P}$.
Recall that $\dim \AM_n = \dim{\mathcal{HP}_n} = \frac{n(n-1)}{2} - 1$, and $\dim \SM_n = \dim{\mathcal{SP}_n} = \frac{n(n+1)}{2} - 1$.
Therefore the following definition captures the proportion of symmetric matrices (where the sum of entries is fixed) which have the Monge property:

\begin{definition}
\label{def:norm vol}
    The \emph{normalized volume} of the symmetric (resp., hollow symmetric) Monge polytope is the ratio
    \[
    \widehat{{\rm Vol}}(\AM_n) \coloneqq \frac{{\rm Vol}(\AM_n)}{{\rm Vol}(\mathcal{HP}_{n})}, \qquad \widehat{{\rm Vol}}(\SM_n) \coloneqq \frac{{\rm Vol}(\SM_n)}{{\rm Vol}(\mathcal{SP}_{n})}.
    \]
\end{definition}

We state the theorem below in terms of the \emph{superfactorial}
\[
{\rm sf}(n) \coloneqq \prod_{i=1}^{n} i!,
\]
i.e., the product of the first $n$ factorials.
The name seems to have been introduced by Sloane and Plouffe~\cite{Sloane}; see entry \href{https://oeis.org/A000178}{A000178} in the OEIS.
Our formulas below involve the squares of superfactorials, which are found in OEIS entry 
\href{https://oeis.org/A055209}{A055209}.

\begin{theorem}
\label{thm:prob HM}
    The normalized volumes of the (hollow) symmetric Monge polytopes are given by
    \[
    \widehat{{\rm Vol}}(\AM_n) = 
    \frac{1}{{\rm sf}(n-1)^2}, \qquad \widehat{{\rm Vol}}(\SM_n) = 
    \frac{1}{{\rm sf}(n-1)^2 \cdot n^n}.
    \]
\end{theorem}

\begin{proof}
    We will give full details in the proof for $\AM_n$, and then afterwards give the very few modifications required for $\SM_n$.
    In the spirit of analytic combinatorics, the proof proceeds from the discrete setting (i.e., integer matrices) to the desired normalized volume, via asymptotics of generating functions.
    
    Let ${\rm H}_n(\mathbb{N})$ be the set of all hollow symmetric $n \times n$ matrices with entries in $\mathbb{N} \coloneqq \mathbb{Z}_{\geq 0}$.
    Denote by ${\rm HM}_n(\mathbb{N}) \subset {\rm H}_n(\mathbb{N})$ the subset consisting of Monge matrices.
    Let $h(-)$ denote half the sum of the entries in a matrix, and partition our sets of matrices accordingly:
    \begin{align*}
        {\rm H}^k_n(\mathbb{N}) & \coloneqq \Big\{X \in {\rm H}_n(\mathbb{N}) : h(X) = k \Big\}, \\
        {\rm HM}^k_n(\mathbb{N}) & \coloneqq \Big\{C \in {\rm HM}_n(\mathbb{N}) : h(C) = k \Big\}.
    \end{align*}
    We observe that
    \begin{equation}
    \label{prob is limit}
        \widehat{{\rm Vol}}(\AM_n) = \lim_{k \rightarrow \infty} \frac{\#{\rm HM}_n^k(\mathbb{N})}{\# {\rm H}_n^k(\mathbb{N})}.
    \end{equation}
    This is because if $X \in {\rm H}_n^k(\mathbb{N})$, then the matrix $\frac{1}{2k} \cdot X \in \mathcal{HP}_n$ has rational entries.
    Thus for each $k$, the ratio on the right-hand side of~\eqref{prob is limit} is the proportion of Monge matrices among certain rational points in $\mathcal{HP}_n$.
    Then by continuity and by density of the rational matrices, the limit in~\eqref{prob is limit} exists and gives the desired proportion of volumes defined in Definition~\ref{def:norm vol}.
    
    Let $m \coloneqq \frac{n(n-1)}{2}$ denote the number of degrees of freedom in a matrix in ${\rm H}_n(\mathbb{N})$, corresponding to the strictly upper-triangular entries.
    Observe that $\#{\rm H}_n^k(\mathbb{N})$ equals the number of ways to fill these $m$ matrix positions with a total of $k$ units, which is the number of weak compositions of $k$ into $m$ parts.
    This number is well known to be the binomial coefficient $\binom{k+m-1}{m-1}$.
    By~\eqref{composition asymptotic} we thus have the following asymptotic in $k$:
    \begin{equation}
        \label{size Hn k}
        \#{\rm H}_n^k(\mathbb{N}) = \binom{k + m -1}{m-1} \sim \frac{k^{m-1}}{(m-1)!}.
    \end{equation}

    Next we will find an asymptotic for $\#{\rm HM}_n^k(\mathbb{N})$.
    The discrete analogue of the proof of Proposition~\ref{prop:SM vertices} implies that for each $C \in {\rm HM}_n(\mathbb{N})$, there is a unique tuple $\lambda \coloneqq \{ \lambda_{ab} \}_{a+b\leq n}$ in $\mathbb{N}^m$, such that
    \[
    C = \sum_{a+b\leq n} \lambda_{ab} \, \NESW{a \times b};
    \]
     conversely, each $\lambda \in \mathbb{N}^m$ determines a unique matrix in ${\rm HM}_n(\mathbb{N})$ in this way.
     Since $h(\NESW{a \times b}) = ab$, we have
    \[
    h(C) = \sum_{a+b \leq n} \lambda_{ab} \cdot ab.
    \]
    Therefore we have the following generating function for ${\rm HM}_n(\mathbb{N})$ with respect to half-sum:
    \begin{align*}
        f(t) &\coloneqq \prod_{a+b\leq n} \frac{1}{1-t^{ab}} \\
        &= \sum_{\lambda \in \mathbb{N}^m} \left(\prod_{a+b\leq n} (t^{ab})^{\lambda_{ab}}\right)\\
        &=        
        \sum_{C \in {\rm HM}_n(\mathbb{N})} \hspace{-10pt}t^{h(C)}\\
        &= \sum_{k = 0}^\infty \underbrace{\# {\rm HM}^k_n(\mathbb{N)}}_{f_k} \cdot \; t^k.
    \end{align*}
It remains to determine the asymptotic behavior of the coefficients $f_k$.
We first observe that $f(t)$ satisfies the hypotheses of Corollary~\ref{cor:asymptotic special}, where $N(t) = 1$ and $D(t) = \prod_{a+b \leq n} (1-t^{ab})$.
By that corollary, we have
\begin{equation}
\label{HM n k asymptotic}
    \#{\rm HM}_n^k(\mathbb{N}) = f_k = \frac{1}{(m-1)!\prod_{a+b \leq n} ab} k^{m-1} + O(k^{m-2}).
\end{equation}
The product in the denominator can be rewritten as a superfactorial:
\[
    \prod_{a+b \leq n} ab = \prod_{a,b=1}^{n-1} a!b! = \prod_{i=1}^{n-1} (i!)^2 = {\rm sf}(n-1)^2.
\]
Substituting this in~\eqref{HM n k asymptotic}, we have the following asymptotic in $k$:
\begin{equation}
    \label{HM final}
    \#{\rm HM}_n^k(\mathbb{N}) \sim \frac{k^{m-1}}{(m-1)! \: {\rm sf}(n-1)^2}.
\end{equation}
Finally, we evaluate the limit in~\eqref{prob is limit} via the asymptotics of its numerator~\eqref{HM final} and denominator~\eqref{size Hn k}, which completes the proof for $\AM_n$:
\[
\widehat{{\rm Vol}}(\AM_n) = \frac{k^{m-1}}{(m-1)! \: {\rm sf}(n-1)^2} \cdot \frac{(m-1)!}{k^{m-1}} = \frac{1}{{\rm sf}(n-1)^2}.
\]

The proof for $\SM_n$ is nearly identical, with the following adjustments.
    Let ${\rm S}_n(\mathbb{N})$ be the set of all symmetric $n \times n$ matrices with entries in $\mathbb{N}$, and with diagonal entries in $2\mathbb{N}$.
    Let ${\rm SM}_n(\mathbb{N}) \subset {\rm H}_n(\mathbb{N})$ be the subset consisting of Monge matrices.
    Again let $h(-)$ denote half the sum of the entries of a matrix, and set
    \begin{align*}
        {\rm S}^k_n(\mathbb{N}) & \coloneqq \Big\{X \in {\rm S}_n(\mathbb{N}) : h(X) = k \Big\}, \\
        {\rm SM}^k_n(\mathbb{N}) & \coloneqq \Big\{C \in {\rm SM}_n(\mathbb{N}) : h(C) = k \Big\}.
    \end{align*}
    By the same reasoning as before, the normalized volume can be computed via
    \begin{equation}
    \label{SM prob is limit}
        \widehat{{\rm Vol}}(\SM_n) = \lim_{k \rightarrow \infty} \frac{\#{\rm SM}_n^k(\mathbb{N})}{\# {\rm S}_n^k(\mathbb{N})}.
    \end{equation}
    Each matrix $X \in {\rm S}_n(\mathbb{N})$ can be written as a unique $\mathbb{N}$-combination of the matrices $E_{ij}+E_{ji}$, for all $1 \leq i \leq j \leq n$.
    There are $m \coloneqq n(n+1)/2$ such matrices, and each contributes $1$ to $h(X)$.
    Therefore we again have
    \begin{equation}
        \label{S asymptotic}
    \#{\rm S}_n^k(\mathbb{N}) = \binom{k + m -1}{m-1} \sim \frac{k^{m+1}}{(m-1)!}.
    \end{equation}
    Next, if $C \in {\rm SM}_n(\mathbb{N})$, then there are unique nonnegative integers $\lambda_{ab}$ and $\lambda_i$, for $a+b\leq n$ and for $1 \leq i \leq n$, such that
    \[
    C = \sum_{a+b\leq n} \lambda_{ab} \NESW{a \times b} + \sum_{i=1}^n \lambda_i \HV{i}.
    \]
Since each copy of $\NESW{a \times b}$ contributes $ab$ to $h(C)$, and each copy of $\HV{i}$ contributes $n$, we have
\[
    h(C) = \sum_{a+b \leq n} \lambda_{ab} \cdot ab + n\sum_{i=1}^n \lambda_i.
\]
Therefore we have the generating function
\[
    f(t) \coloneqq \prod_{a+b\leq n} \frac{1}{1-t^{ab}} \cdot \frac{1}{(1-t^{n})^n}  = \sum_{k = 0}^\infty \underbrace{\# {\rm SM}^k_n(\mathbb{N)}}_{f_k} \cdot \; t^k.
\]
The rest of the proof is identical to the $\AM_n$ case above.
\end{proof}

\subsection{The metric face of the hollow symmetric Monge polytope}
\label{sub:metric face}

In general, any cost matrix $C \in \AM_n$ (or scalar multiple thereof) induces a premetric $d: [n] \times [n] \longrightarrow \mathbb{R}_{\geq 0}$, given by $d(i,j) = c_{ij}$.
(By ``premetric,'' we mean a metric space without the axiom $i \neq j \Longrightarrow d(i,j) \neq 0$ and without the triangle inequality.)
Each vertex of $\AM_n$ has the following straightforward interpretation in the context of the TP (where the source sites $1, \ldots, n$ are identified with the target sites, since we are considering symmetric cost matrices) or the TSP (where we consider the sites $1, \ldots, n$ to be cities): any scalar multiple $\lambda \NESW{a \times b}$ imposes the cost $\lambda$ for transportation between the first $a$ sites and the last $b$ sites, and cost $0$ otherwise.

\begin{definition}
    The \emph{metric face} of the polytope $\AM_n$ is the convex hull of the vertices $\NESW{a \times b}$ for all $a + b = n$.
\end{definition}

The $n$ vertices determining the metric face of $\AM_n$ are maximal, in the sense that their \emph{support} (i.e., the set of matrix coordinates with nonzero entries) is not contained in the support of any other vertex.
The name ``metric face'' is justified by the following proposition:

\begin{prop}
    A cost matrix in $\AM_n$ induces a true metric if and only if it lies in the interior of the metric face.
\end{prop}

\begin{proof}
    Let $C \in \AM_n$, written as a unique sum~\eqref{C sum lambdas} of vertices; we first show that the triangle inequality holds if and only if $C$ lies on the metric face.
    If $C$ lies on the metric face, then it is the convex combination of vertices of the form $\NESW{a \times (n-a)}$.
    At such a vertex, each of the first $a$ sites is the same positive distance away from each of the remaining $n-a$ sites, while all other distances are zero.
    Hence it is easy to see that the triangle inequality is satisfied at such a vertex, and therefore at any convex combination of such vertices.
    Thus if $C$ is on the metric face, then $C$ has the triangle inequality.

    To prove the converse, note that if $C$ is on the metric face, then for $j-i>1$ we have
    \[
    c_{ij} = \sum_{i \leq k < j} c_{k,k+1} = c_{i,i+1} + \sum_{i < k < j} c_{k,k+1} = c_{i,i+1} + c_{i+1,j}.
    \]
    Hence it is impossible to increase $c_{ij}$ without violating the triangle inequality, and thus $\lambda_{i,n-j+1} = 0$.
    Hence if the triangle inequality holds, then $\lambda_{ab} = 0$ for all $a+b < n$, so $C$ lies on the metric face of $\AM_n$.

    Finally, it is clear that the axiom $i \neq j \Longrightarrow c_{ij} \neq 0$ holds if and only if $\lambda_{ab} > 0$ for all $a+b = n$, which requires $C$ to lie in the \emph{interior} of the metric face.
\end{proof}

\section{Main results: the generic Monge polytope}
\label{sec:generic results}

Unlike the symmetric Monge polytopes in the previous section, the generic Monge polytope $\M_{p \times q}$ is not a simplex, and its structure is somewhat complicated.

\subsection{Vertices, facets, and edges of $\M_{p \times q}$}

We begin by explicitly describing and enumerating the vertices, facets, and edges of the Monge polytope.

\begin{prop}
\label{prop:Mpq vertices}
    The vertex set of $\M_{p \times q}$ is
    \begin{align*}
    V(\M_{p \times q}) = \qquad &\left\{ \widehat{\NE{a \times b}} : 1 \leq a < p, \: 1 \leq b < q \right\} \\
    \cup \;\; &\left\{ \widehat{\SW{a \times b}} : 1 \leq a < p, \: 1 \leq b < q \right\}\\
    \cup \;\; &\left\{ \widehat{\HOR{i}} : 1 \leq i \leq p \right\} \\
    \cup \;\; &\left\{ \widehat{\VER{j}} : 1 \leq j \leq q \right\}.
    \end{align*}
    Therefore $\#V(\M_{p \times q}) = 2(p-1)(q-1) + p + q$.
\end{prop}

\begin{proof}

A main result of~\cite{Rudolf}*{Lemma 2.5} states that the nonnegative scalar multiples of the claimed vertices above form extremal rays of the cone of nonnegative $p \times q$ Monge matrices.
This implies both conditions for the vertex set of the polytope $\M_{p \times q}$, as follows.
On one hand, any matrix in the Monge cone can be written as a nonnegative linear combination of these claimed vertices; upon restriction to $\P_{p \times q}$, this implies that any matrix in $\M_{p \times q}$ can be written as a \emph{convex} combination of these claimed vertices.
On the other hand, none of the claimed vertices can be written as a nonnegative linear combination (much less a convex combination) of the other claimed vertices.
Hence the convex hull of any proper subset of the claimed vertices is not all of $\M_{p \times q}$.
\end{proof}

\begin{rem}
\label{rem:OEIS vertices square}
    In the case of square matrices, Proposition~\ref{prop:Mpq vertices} yields $\#V(\mathcal{M}_{p \times p}) = 2[(p-1)^2 + p]$.
    These numbers appear in OEIS entry \href{https://oeis.org/A051890}{A051890}.
    In fact, $\#V(\mathcal{M}_{p \times p})$ equals the number of regions obtained from $p$ ellipses in the plane (where any two ellipses meet in four points).
    The OEIS also describes some intriguing instances of these numbers in neutron shell filling.
\end{rem}

\begin{theorem}
\label{thm:facets generic}
    The facets of $\M_{p\times q}$ are the following:
    \begin{enumerate}
        \item For each $(i,j)$ with $1 \leq  i < p$ and $1 \leq j < q$, there is a facet $F_{ij}$ which is the convex hull of the set
        \[
        V(\M_{p \times q}) \Big\backslash \Big\{\NE{i \times (q-j)}, \; \SW{(p-i) \times j}\Big\}.
        \]
        \item For each $(i,j)$ with $1 \leq i \leq p$ and $1 \leq j \leq q$, there is a facet $G_{ij}$ which is the convex hull of the set
        \[
        \{ \NE{a \times b} : \text{$a < i$ or $b \leq q-j$}\} \cup \{ \SW{a \times b} : \text{$a \leq p-i$ or $b < j$}\} \cup \{\HOR{a} : a \neq i\} \cup \{\VER{b} : b \neq j \}.
        \]
    \end{enumerate}
    Therefore the number of facets is $(p-1)(q-1)+pq$.
\end{theorem}

\begin{proof}
    An H-representation of $\M_{p \times q}$ is given by the system consisting of the Monge inequalities~\eqref{Monge defining inequality} and the nonnegative conditions~\eqref{nonnegative}.
    This representation is not minimal, however, since many of the Monge inequalities~\eqref{Monge defining inequality} are redundant.
    In fact, it is straightforward to check~\cite{Burkard}*{eqn.~(6)} that a matrix is Monge if and only if the Monge property holds for adjacent rows and adjacent columns.
    Hence we can replace the inequalities~\eqref{Monge defining inequality} with the following system:
    \begin{equation}
        \label{Monge minimal}
        c_{ij} + c_{i+1, j+1} 
        \leq c_{i, j+1} + c_{i+1,j}, \quad \text{for all $1 \leq i < p$ and $1 \leq j < q$}.
    \end{equation}
    Therefore a minimal H-representation for $\M_{p \times q}$ is given by the $(p-1)(q-1)$ many inequalities in~\eqref{Monge minimal}, and the $pq$ many inequalities in~\eqref{nonnegative}.
    It remains to determine the vertices for which each defining inequality is tight.

    We claim that each face $F_{ij}$ defined in the theorem is the facet on which the corresponding inequality in~\eqref{Monge minimal} is tight.
    To see this, note that~\eqref{Monge minimal} is an equality at every vertex of $\M_{p \times q}$ (see Proposition~\ref{prop:Mpq vertices} above) \emph{except} those vertices for which exactly one of the four neighboring entries in~\eqref{Monge minimal} is nonzero.
    But there are only two such vertices, namely, those where $c_{i,j+1}$ (resp., $c_{i+1,j}$) is the only nonzero entry of the four; these two vertices are the two vertices excluded from $F_{ij}$.
    Therefore each $F_{ij}$ is a facet.

    Finally, to prove that each face $G_{ij}$ is the facet on which the corresponding inequality in~\eqref{nonnegative} is tight,
    we observe that each inequality in~\eqref{nonnegative} is tight at precisely those vertices $C$ such that $c_{ij} = 0$.
    It is then easy to see that the vertices defining $G_{ij}$ are precisely the vertices for which $c_{ij} = 0$.
\end{proof}

By combining Proposition~\ref{prop:Mpq vertices} and Theorem~\ref{thm:facets generic}, it is straightforward to write down (or at least to program) the \emph{vertex--facet incidence matrix} of $\M_{p \times q}$, where the rows and columns correspond to the vertices and facets, respectively, and each entry is either $1$ (if the facet contains the vertex) or $0$ (otherwise).
From this matrix, then, it is possible to construct the face lattice of $\M_{p \times q}$.
\textit{A fortiori}, one can determine a list of edges by inspecting the vertex--facet incidence matrix, and finding vertex pairs $\{v,w\}$ such that there is no third vertex $x$ which is contained in every facet containing $\{v,w\}$.
In spite of this, we believe that an \emph{explicit} description and enumeration of the edges is still valuable, for example in describing the \emph{polytope graph} of $\M_{p \times q}$ (i.e., the graph of the 1-skeleton, i.e., just the vertices and edges).
Hence we include the following theorem describing the set of edges, which may be omitted without loss of continuity:

\begin{theorem}
\label{thm:edges generic}

    The edges of $\M_{p \times q}$ are the convex hulls of the following vertex pairs:

    \begin{enumerate}
        \item all pairs $\{ \VER{i}, \NE{a \times b} \}$, except those for which $i = p$ and $a = p - 1$;
        \item all pairs $\{ \VER{i}, \SW{a \times b} \}$, except those for which $i = 1$ and $a = p - 1$;
        \item all pairs $\{ \HOR{j}, \NE{a \times b} \}$, except those for which $j = 1$ and $b = q - 1$;
        \item all pairs $\{ \HOR{j}, \SW{a \times b} \}$, except those for which $j = q$ and $b = q - 1$;
        \item all pairs in $\{ \NE{a \times b} : 1 \leq a < p, \: 1 \leq b < q\}$;
        \item all pairs in $\{ \SW{a \times b} : 1 \leq a < p, \: 1 \leq b < q\}$;
        \item those pairs $\{ \NE{a \times b}, \SW{a' \times b'} \}$ such that
        \begin{itemize}
            \item $a + a' < p$, or
            \item $b + b' < q$, or 
            \item $a + a' = p$ and $b+b' = q$.
        \end{itemize}
        \item all pairs in $\{ \HOR{i} : 1 \leq i \leq p \} \cup \{ \VER{j} : 1 \leq j \leq q \}$; but 
        \begin{itemize} 
        \item if $p=2$, then $\{ \HOR{1}, \HOR{2}\}$ is not an edge;
        \item if $q=2$, then $\{ \VER{1}, \VER{2} \}$ is not an edge.
        \end{itemize}
    \end{enumerate}
    
    For $p,q > 2$, the  number of edges in $\M_{p \times q}$ is
    \[
    \frac{1}{4} \Big[ 24 - 18 (p+q) + 19 p q + 2 (p^2 + q^2) - 7 (p^2 q + p q^2 - p^2 q^2) \Big].
    \]
    (Subtract one for each parameter $p$ or $q$ that equals 2.)
\end{theorem}

\begin{proof}
    Let $v$ and $w$ be distinct vertices of a polytope.
    Recall that ${\rm conv}\{v,w\}$ is an edge, if and only if no convex combination of $v$ and $w$ can be written as a convex combination of points which do not all lie in ${\rm conv}\{v,w\}$.
    In the case of our polytope $\M_{p \times q}$, this property can be easily detected by inspecting the support of the matrix $u + v$.
    Noting that the itemized list of eight pair types above comprises all possible combinations of vertices, we verify that the specified pairs are indeed the pairs that form edges:

    In item (1), the combined support of these pairs is unique, except for the following pairings: if we pair the bottommost horizontal strip $\HOR{p}$ with a vertex $\NE{(p-1) \times b}$, then we obtain the same combined support as we do if we combine $\SW{1 \times (q-b)}$ with $\VER{(q-b+1)}, \VER{(q-b+2)}, \ldots, \VER{q}$.
    The argument is identical for items (2)--(4).

    In item (5), where we pair any two northeast blocks, either the combined support is unique, or else one support contains the other, in which case the smaller support is determined by the larger matrix entries in any convex combination.
    The same argument holds for item (6).

    In item (7), the combined support is unique as long as it does not contain a full horizontal or vertical strip.
    Therefore, picturing the two supports as rectangles (one in the northeast corner, one in the southwest), we have an edge as long as the two rectangles do not overlap or even touch (where ``touching'' excludes the case where they touch only at their corners).
    The three bulleted conditions under item (7) are precisely those required for this non-touching condition.

    In item (8), as long as $p,q>2$ it is clear that any two full strips (whether horizontal or vertical or one of each) have a unique combined support.
    If, however, $p=2$, then the combined support of $\HOR{1}$ and $\HOR{2}$ equals the combined support of $\VER{1}, \ldots, \VER{q}$, and vice versa for $q=2$.

    As for the enumeration, it is straightforward to check that
    \begin{itemize}
        \item items (1) and (2) each contribute $p(p-1)(q-1) - (q-1)$ edges;
        \item items (3) and (4) each contribute $q(p-1)(q-1) - (p-1)$ edges;
        \item items (5) and (6) each contribute $\binom{(p-1)(q-1)}{2}$ edges;
        \item by summing over pairs $(a,b)$ and counting the number of valid pairs $(a',b')$ for each $(a,b)$, we see that item (7) contributes the following number of edges:
        \[        \underbrace{\sum_{a,b}}_{\mathclap{\substack{(p-1)(q-1) \\  \text{summands}}}} \Big[\overbrace{(p-1)(q-1)}^{\substack{\text{all pairs} \\ (a',b')}} - \underbrace{ab}_{\mathclap{\substack{\text{support} \\ \text{of $\NE{a \times b}$}}}} + \overbrace{1}^{\mathclap{\substack{\text{allowed to touch} \\ \text{at corners}}}} \Big] = [(p-1)(q-1)]^2 + (p-1)(q-1) - \underbrace{\sum_{a,b} ab}_{\mathclap{\substack{(1+ \ldots + p-1)(1+ \cdots + q-1) \\[1ex] = \frac{p(p-1)}{2} \cdot \frac{q(q-1)}{2}}}}
\]
\item if $p,q>2$, then item (8) contributes $\binom{p+q}{2}$ edges; subtract one edge for each parameter $p,q$ that equals 2.
    \end{itemize}
    Summing these results and expanding, one obtains the expression in the theorem.
\end{proof}

\begin{ex}
    Combining the three main results (vertices, facets, edges) in this subsection, we can determine that the Monge polytope $\M_{2 \times 2}$ has $f$-vector $(1,6,9,5,1)$, with vertices, edges, and facets as depicted below.
    In fact, $\M_{2 \times 2}$ is a triangular prism lying inside the three-dimensional simplex $\P_{2 \times 2}$ embedded in $\mathbb R^4$; see the plot in Figure~\ref{fig:M2}.
    
\begin{center}
\tdplotsetmaincoords{70}{120} 

\begin{tikzpicture}[scale=2.2,tdplot_main_coords]
    \coordinate (A) at (0,0,0);
    \coordinate (B) at (1,0,0);
    \coordinate (C) at (0.5,0.866,0);
    \coordinate (D) at (0,0,1);
    \coordinate (E) at (1,0,1);
    \coordinate (F) at (0.5,0.866,1);
    
    \fill[green!30!gray,opacity=.3] (A) -- (B) -- (C) -- cycle;
    
    \fill[green!30!gray,opacity=.3] (D) -- (E) -- (F) -- cycle;

    \fill[blue!30!gray,opacity=.3] (D) -- (F) -- (C) --  (A) -- cycle;

    \fill[blue!30!gray,opacity=.3] (D) -- (E) -- (B) --  (A) -- cycle;
    
    \draw (A) -- (B);
    \draw (B) -- (C);
    \draw (C) -- (A);
    \draw (D) -- (A);
    \draw (D) -- (E);
    \draw (F) -- (D);
    \draw (B) -- (E);
    \draw(C) -- (F);
    
    \draw[dashed] (E) -- (F);

    \node[anchor=south west] at (A) {$\left[\begin{smallmatrix}0&1\\0&0\end{smallmatrix}\right]$};
    \node[anchor=north east] at (B) {$\left[\begin{smallmatrix}.5&.5\\0&0\end{smallmatrix}\right]$};
    \node[anchor=north west] at (C) {$\left[\begin{smallmatrix}0&.5\\0&.5\end{smallmatrix}\right]$};
    \node[anchor=south] at (D) {$\left[\begin{smallmatrix}0&0\\1&0\end{smallmatrix}\right]$};
    \node[anchor=south east] at (E) {$\left[\begin{smallmatrix}.5&0\\.5&0\end{smallmatrix}\right]$};
    \node[anchor=south west] at (F) {$\left[\begin{smallmatrix}0&0\\.5&.5\end{smallmatrix}\right]$};
\end{tikzpicture}
\end{center}
\end{ex}

\begin{figure}
    \centering
    \includegraphics[width=.6\linewidth]{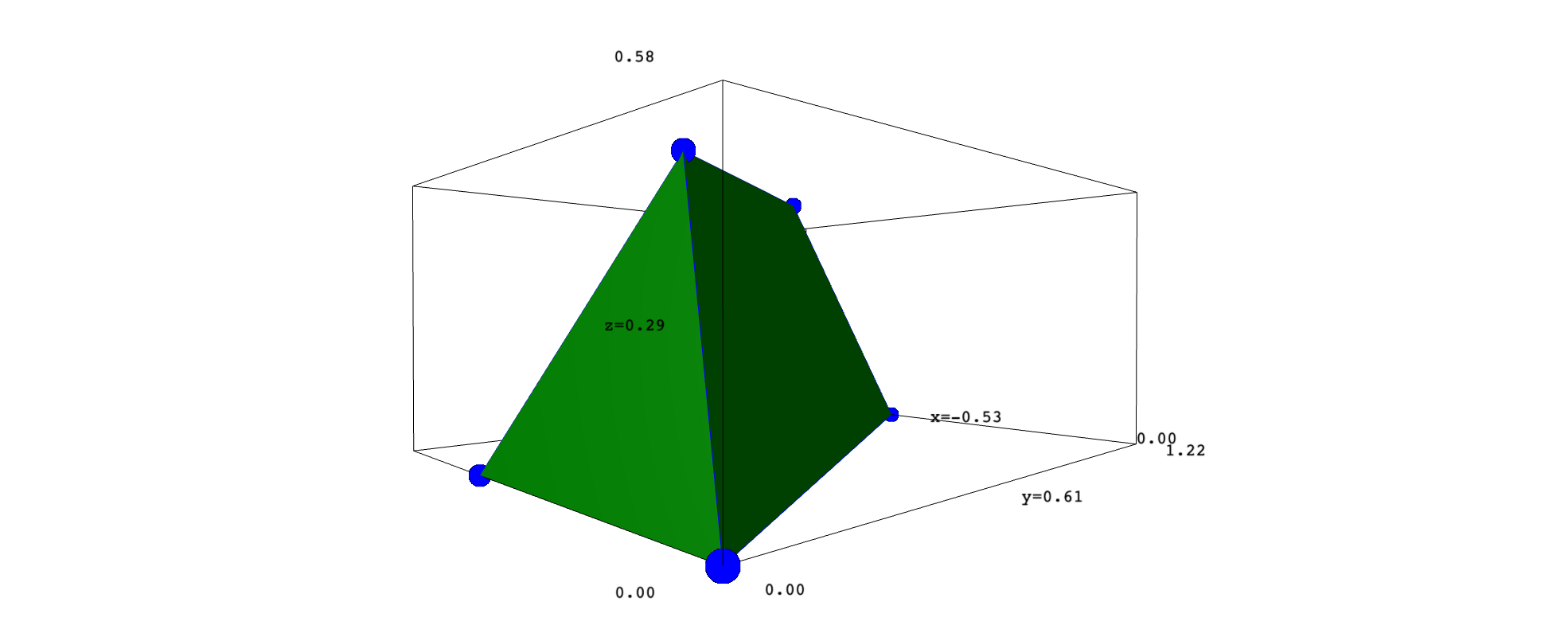}
    \caption{The Monge polytope $\M_{2 \times 2}$ is a triangular prism, with $f$-vector $(1,6,9,5,1)$.
    By Theorem~\ref{thm:volume 2-row}, its normalized volume is $1/4$, which means that it occupies exactly $1/4$ of the probability simplex $\P_{2 \times 2}$.}
    \label{fig:M2}
\end{figure}

\subsection{Normalized volume of $\M_{2 \times p}$}

We define the normalized volume just as in the symmetric case, namely,
\[
\widehat{{\rm Vol}}(\M_{p \times q}) \coloneqq \frac{{\rm Vol}(\M_{p \times q})}{{\rm Vol}(\P_{p \times q})}.
\]
In the case of two-row matrices, there is an especially nice formula for this normalized volume:
\begin{theorem}
\label{thm:volume 2-row}
    We have $\widehat{{\rm Vol}}(\M_{2 \times p}) = \dfrac{1}{p!}$.
\end{theorem}

Our proof of this theorem will be similar to that of Theorem~\ref{thm:prob HM}.
Because nonsymmetric Monge matrices do not admit \emph{unique} convex combinations of vertices, however, it will first take some work to determine the required generating function.
The key will be to write down a Stanley decomposition for the set of $2 \times p$ Monge matrices with nonnegative entries, which will establish a canonical form for writing Monge matrices as linear combinations of certain fundamental matrices.
From this will follow the desired generating function and the proof of Theorem~\ref{thm:volume 2-row}.

A \emph{Stanley decomposition} of a commutative monoid $M$ is a finite disjoint union
\begin{equation}
    \label{Stanley decomp general}
    M = \bigsqcup_{i = 1}^\ell \delta_i + \mathbb{N}[\gamma_{i,1}, \ldots, \gamma_{i,m}]
\end{equation}
where each component is the translate (by $\delta_i$) of a free commutative submonoid generated by elements $\gamma_{ij}$.  (See~\cite{Stanley}*{Thm.~5.2}.)
A Stanley decomposition~\eqref{Stanley decomp general} leads to a canonical form on $M$ in the following sense: for each $\mu \in M$, there exists a unique $i \in \{1, \ldots, \ell\}$ and unique integers $a_1, \ldots, a_m \in \mathbb{N}$ such that
\[
    \mu = \delta_i + \sum_{j=1}^m a_j \gamma_{i,j}.
\]

For our present purpose, we take as our monoid ${\rm MM}_{2 \times p}(\mathbb{N})$, the set of $2 \times p$ Monge matrices with entries in $\mathbb{N}$, under the operation of matrix addition.
Let $[p] \coloneqq \{1, \ldots, p\}$.
We claim that there is a Stanley decomposition of the form
\begin{equation}
    \label{Stanley decomp MM2p}
    {\rm MM}_{2 \times p}(\mathbb{N}) = \bigsqcup_{S \subseteq [p]} \delta_S + \mathbb{N}\Big[ \VER{1}, \ldots, \VER{p}, \gamma_{S,1}, \ldots, \gamma_{S,p} \Big],
\end{equation}
where $\delta_S$ and the $\gamma_{S,j}$ are all constructed from a certain matrix $\Gamma_S$, as follows.

Each $S \subseteq [p]$ determines a $2 \times p$ Monge matrix $\Gamma_S$, obtained by filling the right end of the top row with the elements of $S$, and filling the left end of the bottom row with the elements in the complement $\overline{S} \subseteq [p]$.
(Note that the Monge condition requires that the top entries increase from left to right, while the bottom entries increase from right to left.)
For example, if $p=12$ and $S = \{3,4,5,7,10\}$, then
\[
\Gamma_S = \begin{bmatrix}
    0&0&0&0&0&0&0&3&4&5&7&10\\
    12&11&9&8&6&2&1&0&0&0&0&0
\end{bmatrix}.
\]

The element $\delta_S$ is obtained from $\Gamma_S$ by first zeroing out any entries in the bottom row which are smaller than the smallest entry in the top row; then in each of the two rows, replace the $i$th string of consecutive numbers (in ascending order) by the string of all $i$'s.
In the example above, we zero out entries 1 and 2, and then obtain
\[
\delta_S = \begin{bmatrix}
    0&0&0&0&0&0&0&1&1&1&2&3\\
    3&3&2&2&1&0&0&0&0&0&0&0
\end{bmatrix}
\]

For each $j = 1, \ldots, p$, the generator $\gamma_{S,j}$ is the binary matrix obtained from $\Gamma_S$ by replacing each entry $i$ by
\[
\begin{cases}
    1,& i > p-j,\\
    0 & \text{otherwise}.
\end{cases}
\]
In the example above, for instance, we have
\[
\gamma_{S,8} = \begin{bmatrix}
    0&0&0&0&0&0&0&0&0&1&1&1\\
    1&1&1&1&1&0&0&0&0&0&0&0
\end{bmatrix}.
\]
Note that there are exactly $j$ nonzero entries in $\gamma_{S,j}$, and $\gamma_{S,j+1}$ is obtained by changing a single 0 to a 1.
We have defined the $\gamma_{S,j}$'s precisely so that the entries of any matrix in $\delta_S + \mathbb{N}[\gamma_{S,1}, \ldots, \gamma_{S,p}]$ preserve the same ordering as their corresponding positions in $\Gamma_S$, where equality is allowed (1) among the positions corresponding to strings of equal numbers (in each row separately) in $\delta_S$, or (2) between top and bottom entries if and only if the corresponding bottom position in $\Gamma_S$ is greater than the corresponding top position in $\Gamma_S$.

\begin{lemma}
    \label{lemma:Stanley decomp}
    The monoid ${\rm MM}_{2 \times p}(\mathbb N)$ admits the Stanley decomposition~\eqref{Stanley decomp MM2p}, where $\delta_S$ and $\gamma_{S,j}$ are defined as above.
    Therefore we have the generating function
    \begin{equation}
        \label{gen func 2-row}
        f(t) \coloneqq \sum_{C \in {\rm MM}_{2 \times p}(\mathbb{N})} t^{|C|} = \frac{\sum_{S \subseteq [p]}t^{|\delta_S|}}{(1-t^2)^p \prod_{j=1}^p (1-t^j)},
    \end{equation}
    where $|C| \coloneqq \sum_{i,j} c_{ij}$ is the sum of the entries of $C$.
\end{lemma}

\begin{proof}
    We need to show that each $2 \times p$ nonnegative integer Monge matrix lies in exactly one of the components in~\eqref{Stanley decomp MM2p}, indexed by a unique subset $S \subseteq [p]$.
    To do this, we give an algorithm for writing any matrix $C \in {\rm MM}_{2 \times p}(\mathbb{N})$ in its canonical form
    \[
    C = \delta_S + \sum_{j=1}^p a_j \VER{j} + \sum_{j=1}^p b_j \gamma_{S,j}.
    \].

    \begin{enumerate}
    \item Each $a_j$ is the minimum entry in the $j$th column of $C$.
    
    \item Set $C' \coloneqq C - \sum_{j=1}^p a_j \VER{j}$.

    \item Convert $C'$ into the matrix $\Gamma_S$ (which will determine $\delta_S$) as follows:

    \begin{enumerate}
        \item If there are, say, $z$ columns of $C'$ with two zeros, fill the bottom row in those columns with $1, 2, \ldots, z$ from right to left.

        \item Now replace the original (nonzero) entries in $C'$, in order from smallest to largest, with the numbers $z+1, z+2, \ldots, p$.
        To break ties, replace entries from left to right in the top row, and from right to left in the bottom row.
        If there is a tie between the top and bottom row, then replace the top entry.
    \end{enumerate}

    \item Set $C'' \coloneqq C' - \delta_S$.

    \item For each $j \in [p]$, let $(j)$ be the position in $\Gamma_S$ with entry $j$.
    Then $b_j = c''_{(j)} - c''_{(j-1)}$.
\end{enumerate}
This algorithm is justified as follows.
In steps (1) and (2), we carry out the unique way of ``stripping'' off multiples of the generators $\VER{i}$, which are common to every component of~\eqref{Stanley decomp MM2p}, until we obtain a matrix $C'$ with a zero in every column.
The $z$ columns (if any) containing \emph{two} zeros tell us that $\overline{S}$ must contain $1, \ldots, z$, which follows from our construction of $\delta_S$ above; hence step (3a) fills in these entries in $\Gamma_S$.
By our observation before the lemma, the ordering of the entries in $C'$ uniquely determines the matrix $\Gamma_S$, as described in step (3b).
Now having determined $\Gamma_S$, we know that $C'$ lies in the component of~\eqref{Stanley decomp MM2p} indexed by the subset $S$, and so in step (4) we subtract $\delta_S$, leaving just a combination $C''$ of the generators $\gamma_{S,j}$.
Since each $\gamma_{S,j+1}$ is obtained from $\gamma_{S,j}$ by changing a single 0 to a 1, and since the location of this new 1 is indexed by the entries of $\Gamma_S$, it is straightforward in step (5) to find the coefficient $b_j$ of each $\gamma_{S,j}$, which must be the difference between the entries in $C''$ which lie in positions with consecutive entries in $\Gamma_S$.

Now that we have shown that~\eqref{Stanley decomp MM2p} is a Stanley decomposition, the generating function~\eqref{gen func 2-row} is straightforward.
The free submonoid in each component of~\eqref{Stanley decomp MM2p} has $2p$ generators, where each $|\VER{j}|= 2$ and where $|\gamma_{S,j}| = j$ for $1 \leq j \leq p$; hence the generating function for each individual component is $t^{|\delta_S|}/(1-t^2)^p\prod_{j=1}^p (1-t^j)$.
Because the components in~\eqref{Stanley decomp MM2p} are all disjoint, the generating function for their union is just the sum of these individual generating functions.    
\end{proof}

\begin{rem}
    The numerator of the generating function~\eqref{gen func 2-row} has the combinatorial interpretation $\sum_\lambda t^{|\lambda|}$, ranging over all partitions $\lambda$ whose Young diagram has maximum hook length at most $p$.
    (The maximum hook is simply the union of the first row and first column of the Young diagram.) 
    To translate between the matrices $\delta_S$ and the corresponding Young diagrams, one converts the entries of $\delta_S$ into the diagonal lengths of $\lambda_S$: in particular, the $i$th largest entry in the top row of $\delta_S$ gives the length of the $i$th diagonal of $\lambda_S$, beginning with the main diagonal and moving right.
    Likewise, the $i$th largest entry in the bottom row of $\delta_S$ gives the length of the $i$th subdiagonal of $\lambda_S$, beginning just below the main diagonal and moving downward.
    For example, resuming the example above, one has
    \[
    \ytableausetup{smalltableaux,centertableaux}
    \delta_S = \begin{bmatrix}
    0&0&0&0&0&0&0&1&1&1&2&3\\
    3&3&2&2&1&0&0&0&0&0&0&0
\end{bmatrix} \quad \leadsto \quad \lambda_S = \ydiagram{5,3,3,3,3,2}.
    \]
    This construction makes it obvious that $|\delta_S| = |\lambda_S|$, and also that the largest hook length in $\lambda_S$ is the number of nonzero entries in $\delta_S$, which is at most $p$.
    The construction is easily invertible, so that every Young diagram with maximum hook length at most $p$ corresponds to the matrix $\delta_S$ for a unique $S \subseteq [p]$.
    Moreover, we point out that this polynomial $\sum_\lambda t^{\lambda}$ is the same as the polynomial $\widetilde{T}_{p-1}(t)$ defined in the nuclear physcis paper~\cite{Isachenkov}*{eqn.~(4.6)}.
    See also OEIS entry \href{https://oeis.org/A161161}{A161161}.  
\end{rem}

\begin{proof}[Proof of Theorem~\ref{thm:volume 2-row}]
    Armed with the generating function~\eqref{gen func 2-row}, we can employ the same method as in the proof of Theorem~\ref{thm:prob HM}.
    Letting ${\rm M}^k_{2 \times p}(\mathbb{N})$ denote the set of all $2 \times p$ matrices with entries in $\mathbb{N}$ summing to $k$, and ${\rm MM}_{2 \times p}^k(\mathbb{N})$ the subset of Monge matrices, we have
    \[
    \widehat{\rm Vol}(\M_{2 \times p}) = \lim_{k \rightarrow \infty} \frac{\#{\rm MM}_{2 \times p}^k (\mathbb{N})}{\#{\rm M_{2 \times p}^k (\mathbb{N})}}.
    \]
    For the denominator, we have 
    \[
    \#{\rm M}_{2 \times p}^k(\mathbb{N}) = \binom{k+ 2p - 1}{2p - 1} \sim \frac{k^{2p-1}}{(2p-1)!}.
    \]
    For the numerator, we apply Corollary~\ref{lemma:asymptotic} to the generating function~\eqref{gen func 2-row}, where $N(t) = \sum_S t^{|\delta_S|}$ and $D(t) = (1-t^2)^p \prod_{j=1}^p (1-t^j)$, and the root $\alpha = 1$ has multiplicity $2p$.
    Since there are $2^p$ subsets $S \subseteq [p]$, we have $N(1) = 2^p$.
    Hence Corollary~\ref{lemma:asymptotic} yields
    \[
    \#{\rm MM}_{2 \times p}^k (\mathbb{N}) \sim \frac{2^p}{(2p-1)!p! 2^p}k^{2p - 1}.
    \]
    The result follows upon canceling repeated factors.
\end{proof}

For matrices with more than two rows or columns, it becomes much more difficult to write down a Stanley decomposition because the same Monge matrix can be written as the sum of vertices in many more different ways.
Regardless, it seems to us that any such decomposition would not be \emph{pure} (in the sense that the free submonoids in each component have the same number of generators), which defeats the purpose of using the Stanley decomposition to write the generating function in a nice rational form.
There are certainly ad hoc methods for determining the generating function (and therefore the normalized volume), such as interpreting the matrices as degree matrices for monomials in the variables $x_{ij}$, and then using software (such as Macaulay2) to find the Hilbert series of the ring generated by the monomials of the vertices.
Using this method, one obtains the following normalized volumes:
\[
    \widehat{{\rm Vol}}(\M_{3 \times 3}) = \frac{17}{1296}, \qquad
    \widehat{{\rm Vol}}(\M_{3 \times 4}) = \frac{7}{10368}, \qquad
    \widehat{{\rm Vol}}(\M_{4 \times 4}) =\frac{361}{95551488}.
\]

\bibliographystyle{alpha}
\bibliography{references}

\end{document}